\documentclass[reqno,b5paper]{amsart}

\usepackage{amssymb,amsmath}

\newtheorem{theorem}{Theorem}[section]
\newtheorem{lemma}[theorem]{Lemma}
\newtheorem{definition}[theorem]{Definition}
\newtheorem{proposition}[theorem]{Proposition}
\newtheorem{corollary}[theorem]{Corollary}

\newtheorem{example}[theorem]{Example}

\begin{document}

\title[On Gibson functions with connected graphs]{On Gibson functions with connected graphs}

\author[Olena Karlova \and Volodymyr Mykhaylyuk]{Olena Karlova* \and Volodymyr Mykhaylyuk**}

\newcommand{\acr}{\newline\indent}

\address{\llap{*\,} Chernivtsi National University\acr
                    Department of Mathematical Analysis\acr
                    Kotsjubyns'koho 2\acr
                    Chernivtsi 58012\acr
                    UKRAINE}
\email{Maslenizza.Ua@gmail.com}

\address{\llap{**\,} Chernivtsi National University\acr
                    Department of Mathematical Analysis\acr
                    Kotsjubyns'koho 2\acr
                    Chernivtsi 58012\acr
                    UKRAINE}
\email{vmykhaylyuk@ukr.net}

\subjclass[2010]{Primary 26B05, 54C08; Secondary  26A21, 54C10}

\keywords{Gibson function, Darboux function, peripherally continuous functions, function with connected graph, Baire-one function}

\begin{abstract}
 A function $f:X\to Y$ between topological spaces is said to be a {\it weakly Gibson function} if $f(\overline{G})\subseteq \overline{f(G)}$ for any open connected set \mbox{$G\subseteq X$}. We call a function $f:X\to  Y$ {\it segmentary connected} if $X$ is topological vector space and $f([a,b])$ is connected for every segment $[a,b]\subseteq X$. We show that if $X$ is a hereditarily Baire space, $Y$ is a metric space, \mbox{$f:X\to Y$} is a Baire-one function and one of the following conditions holds: (i) $X$ is a connected and locally connected space and $f$ is a weakly Gibson function, (ii) $X$ is an arcwise connected space and $f$ is a Darboux function, (iii) $X$ is a topological vector space and $f$ is a segmentary connected function, then $f$ has a connected graph.
\end{abstract}

\maketitle

\section{Introduction}

In 1907, J.~Young showed \cite{Y} that for a Baire-one function \mbox{$f:{\mathbb R}\to {\mathbb R}$}, the Darboux property is equivalent to the following property: for every $x\in\mathbb R$ there exist $(x_n)_{n=1}^\infty$ and $(y_n)_{n=1}^\infty$ such that $x_n\uparrow x$, $y_n\downarrow x$ and $\lim\limits_{n\to\infty}f(x_n)=\lim\limits_{n\to\infty}f(y_n)=f(x)$. In more general spaces, functions having the property of Young are said to be peripherally continuous (see definition in Section~2).

K.~Kuratowski and W.~Sierpi\'{n}ski in 1922 proved that a Baire-one function \mbox{$f:{\mathbb R}\to {\mathbb R}$} has the Darboux property if and only if $f$ has a connected graph \cite{KS}. Remark that for functions $f:{\mathbb R}\to Y$, where $Y$ is a topological space, having a connected graph is equivalent to being a connectivity function (see definition in Section~2).

M.~Hagan \cite{H} and  G.~Whyburn \cite{W} showed that a function $f:{\mathbb R}^n\to {\mathbb R}$ for $n>1$ is connectivity if and only if $f$ is peripherally continuous.

J.~Farkov\'{a} in \cite{JF} introduced a similar property to the Young property: a function $f:X\to Y$ between metric spaces $X$ and $Y$ has the property $\mathcal A$ if for any open connected non-trivial set $G\subseteq X$ and a point $x\in\overline{G}$ there exists a sequence $(x_n)_{n=1}^\infty$, $x_n\in G\setminus\{x\}$, such that $\lim\limits_{n\to\infty}x_n=x$ and $\lim\limits_{n\to\infty}f(x_n)=f(x)$. She proved that for $X=\mathbb R$ every Darboux function has the property ${\mathcal A}$, but for $X={\mathbb R}^2$ this is not true. Moreover, Farkov\'{a} showed that if $X$ is a connected and locally connected complete metric space, $Y$ is a separable metric space and $f:X\to Y$ is a Baire-one function having the property ${\mathcal A}$, then $f$ has a connected graph.

Recently, K.~Kellum \cite{Kel1} introduced and investigated properties of Gibson and weakly Gibson real-valued  functions (see Definition \ref{d1})  defined on a connected space. Some properties of Gibson functions (in particular, Darboux-like properties or continuity of Baire-one Gibson functions) are also studied by M.J.~Evans and P.D. Humke in \cite{EH}.

In this work we introduce a strong Gibson property (see Definition \ref{d1}) which is equivalent to the property $\mathcal A$  for a function between metric spaces.
In Section~3 we study how weak and strong Gibson properties connect with the property of peripherally continuity and with the Darboux property.
In the forth section we consider Baire-one Gibson functions.
In Section 5 we introduced segmentary connected functions defined on a topological vector space and investigate some of their properties.
 Finally, in the sixth section we prove different generalizations of the mentioned result from \cite{JF}. We show that if $X$ is a hereditarily Baire space, $Y$ is a metric space and \mbox{$f:X\to Y$} is a Bare-one function, then each of the following conditions is sufficient for $f$ to have a connected graph: (i) $X$ is a connected and locally connected space and $f$ is weakly Gibson; (ii) $X$ is an arcwise connected space and $f$ is a Darboux function; (iii) $X$ is a topological vector space and $f$ is a segmentary connected function.

\section{Preliminaries}

We will denote by ${\mathcal T}(X)$ and ${\mathcal C}(X)$ the collection of all open and connected subsets of a topological space $X$, respectively. Let ${\mathcal G}(X)={\mathcal T}(X)\cap {\mathcal C}(X)$.

Let $X$ and $Y$ be topological spaces and let $f:X\to Y$ be a function. Then
  \begin{itemize}

    \item $f$ is a {\it Darboux function}, $f\in {\rm D}(X,Y)$, if $f(C)$ is connected for any $C\in{\mathcal C}(X)$;

   \item $f$ is a {\it weakly Darboux function}, $f\in {\rm wD}(X,Y)$, if $f(G)$ is connected for any $G\in {\mathcal G}(X)$;

    \item $f$ {\it has a connected graph}, $f\in {\rm CG}(X,Y)$, if the set ${\rm Gr}\,(f)=\{(x,f(x)):x\in X\}$ is connected;

    \item $f$ is a {\it connectivity function}, $f\in {\rm Conn}(X,Y)$, if the graph of $f|_C$ is connected for any $C\in{\mathcal C}(X)$;

    \item $f$ is {\it peripherally continuous at $x\in X$} if for an arbitrary open neighborhoods $U$ and $V$ of $x$ and $f(x)$, respectively, there exists an open neighborhood $W$ of $x$ such that $W\subseteq U$ and $f({\rm fr}\,W)\subseteq V$, where ${\rm fr}\,W$ denotes the boundary of $W$; $f$ is {\it peripherally continuous on $X$}, $f\in {\rm PC}(X,Y)$, if $f$ is peripherally continuous at every point $x$ of $X$;

    \item $f$ is a {\it Baire-one function}, $f\in {\rm B_1}(X,Y)$, if there exists a sequence of continuous functions $f_n:X\to Y$ such that $\lim\limits_{n\to\infty}f_n(x)=f(x)$ for every $x\in X$;

    \item $f$ is {\it pointwise discontinuous} if the set of all continuity points of $f$ is dense in $X$;

    \item $f$ is {\it barely continuous} if the restriction $f|_F$ has a continuity point for any closed set $F\subseteq X$.

  \end{itemize}

The class of all continuous functions between $X$ and $Y$ we denote by ${\rm C}(X,Y)$.

It is easy to see that
  $$
  {\rm C}(X,Y)\subseteq {\rm Conn}(X,Y)\subseteq {\rm D}(X,Y)\subseteq {\rm wD}(X,Y)
  $$
  for arbitrary topological spaces $X$ and $Y$.

  The basic relations between the classes described above for $X=Y=\mathbb R$ are given in the following chart.
  $$
  {\rm C}({\mathbb R}, {\mathbb R})\subset {\rm Conn}({\mathbb R},{\mathbb R})={\rm CG}({\mathbb R},{\mathbb R})\subset {\rm D}({\mathbb R},{\mathbb R})\subset {\rm PC}({\mathbb R},{\mathbb R}).
  $$
  For the Baire-one functions the chart changes to the following.
  $$
  {\rm C}({\mathbb R}, {\mathbb R})\subset {\rm Conn}({\mathbb R},{\mathbb R})={\rm CG}({\mathbb R},{\mathbb R})={\rm D}({\mathbb R},{\mathbb R})= {\rm PC}({\mathbb R},{\mathbb R}).
  $$
 If  $X={\mathbb R}^n$,  $Y={\mathbb R}$, $n\ge 2$, then
  $$
  {\rm C}({\mathbb R}^n,{\mathbb R})\subset {\rm Conn}({\mathbb R}^n,{\mathbb R})={\rm PC}({\mathbb R}^n,{\mathbb R})\subset {\rm D}({\mathbb R}^n,{\mathbb R}).
  $$
For a thorough treatment we refer the reader to \cite{EGP,CJ,GN}.

Topological space  $X$ is {\it hereditarily Baire} if every closed subset of $X$ is a Baire space.

Topological space $X$ is {\it resolvable} if $X$ is a union of two disjoint everywhere dense subsets.

A subset $A$ of topological vector space $X$ is {\it balanced} if $\lambda A\subseteq A$ whenever $|\lambda|\le 1$.

For a function $f:X\to Y$ between topological spaces by $C(f)$ and $D(f)$ we denote the sets of continuity and discontinuity points of $f$, respectively.

\section{The Gibson properties and their connection with\\ the Darboux-like properties}

\begin{definition}\label{d1}
{\rm  Let $X$ and $Y$ be topological spaces. A function $f:X\to Y$ is

  \begin{itemize}
  \item {\it  weakly Gibson at $x\in X$} if $f(x)\in \overline{f(G)}$ for any \mbox{$G\in{\mathcal G}(X)$} such that $x\in \overline{G}$;

  \item {\it Gibson at $x\in X$} if $f(x)\in \overline{f(G)}$ for any \mbox{$G\in{\mathcal T}(X)$} such that $x\in \overline{G}$;

  \item {\it strongly Gibson at $x\in X$} if $f(x)\in \overline{f(G\cap U)}$ for any $G\in{\mathcal G}(X)$ such that $x\in \overline{G}$ and for any open neighborhood  $U$ of $x$.
  \end{itemize}

  A function $f:X\to Y$ has {\it (weak, strong) Gibson property on $X$} if it has such a property at every point of $X$.
  Classes of all such functions we denote by ${\rm wG}(X,Y)$, ${\rm  G}(X,Y)$ and ${\rm sG}(X,Y)$, respectively.}
\end{definition}

It is evident that
$${\mathrm C}(X,Y)\subseteq {\rm sG}(X,Y)\subseteq {\rm wG}(X,Y)$$ for any topological spaces $X$ and $Y$.

\begin{proposition}
Let $X$ be a locally connected space such that ${\mathcal G}(X)$  is closed under finite intersections and $Y$ be a topological space. Then
$${\rm wG}(X,Y)={\rm sG}(X,Y).$$
\end{proposition}

\begin{proof}
   Let $f\in{\rm wG}(X,Y)$, $x\in X$, $G\in{\mathcal G}(X)$, $x\in\overline{G}$ and $U$ be a neighborhood of $x$. We can choose an open connected neighborhood $V$ of $x$ such that $V\subseteq U$. Then $V\cap G\in{\mathcal G}(X)$ and $x\in\overline{V\cap G}$. Therefore, $f(x)\in\overline{f(V\cap G)}\subseteq\overline{f(U\cap G)}$. Hence, $f\in {\rm sG}(X,Y)$.
\end{proof}

\begin{corollary}
  ${\rm wG}({\mathbb R},Y)={\rm sG}({\mathbb R},Y)$ for an arbitrary $Y$.
\end{corollary}

\begin{proposition}\label{6}
  Let $X$ be a topological space and  $Y$ be a $T_1$-space. Then
  $${\rm D}(X,Y)\subseteq {\rm wG}(X,Y).$$
\end{proposition}
\begin{proof}
 Let $f\in{\rm D}(X,Y)$. Choose  $x\in X$ and $G\in{\mathcal G}(X)$ such that $x\in {\rm fr}\, G$. Fix an arbitrary neighborhood $V$ of  $f(x)$ in $Y$. Since $G$ is connected, $G\cup\{x\}$ is connected too. Then $f(G\cup\{x\})$ is connected since $f$ has the Darboux property. Suppose that $f(G)\subseteq Y\setminus V$. Then
 $$
 f(G\cup\{x\})=f(G)\cup f(x)\subseteq (Y\setminus V)\sqcup\{f(x)\},
 $$
 which contradicts the fact that $f(G\cup\{x\})$ is connected.
 \end{proof}

 We will apply the following results to construct a function {$f:{\mathbb R}^2\to {\mathbb R}$} of the first Baire class such that $f$ has the weak Gibson property, but has no strong Gibson property.

\begin{proposition}\label{13}
  Let $X$, $Y$ and $Z$ be topological spaces, $f\in{\rm wG}(X,Z)$ and $g(x,y)=f(x)$ for all $(x,y)\in X\times Y$. Then $g\in{\rm wG}(X\times Y,Z)$.
\end{proposition}

\begin{proof}
  Fix $G\in{\mathcal G}(X\times Y)$ and $p_0=(x_0,y_0)\in X\times Y$, $p_0\in\overline{G}$. Choose any neighborhood  $W$ of $g(p_0)$ in $Z$. Since a projection is an open map, $U={\rm pr}_{\mathbb R} G$ is an open and connected set in $X$. Notice that $x_0\in\overline{U}$. Since $f$ has the weak Gibson property at $x_0$, there exists $x\in U$ such that $f(x)\in W$. Take an arbitrary $(x,y)\in(\{x\}\times Y)\cap G$.  Then  $g(x,y)=f(x)\in W$. Hence, $g$ has is weakly Gibson at $p_0$.
\end{proof}

\begin{proposition}\label{14}
  Let $f:{\mathbb R}\to {\mathbb R}$ be a function and $g(x,y)=f(x)$ for all \mbox{$(x,y)\in {\mathbb R}^2$}. If $g\in {\rm sG}({\mathbb R}^2,{\mathbb R})\cap {\rm B_1}({\mathbb R}^2,{\mathbb R})$, then $f\in {\rm C}({\mathbb R},{\mathbb R})$.
\end{proposition}

\begin{proof}
  Suppose the contrary and choose $x_0\in D(f)$. Then there exists such a sequence $(x_n)_{n=1}^\infty$ that $\lim\limits_{n\to\infty}x_n=x_0$ and $\lim\limits_{n\to\infty}f(x_n)\ne f(x_0)$. Without loss of the generality we may assume that $x_n>x_0$ for all $n\in\mathbb N$. Remark that the sets $\Gamma_1={\rm Gr}(f|_{[x_0,+\infty]})$ and $\Gamma_2={\rm Gr}(f|_{(x_0,+\infty]})$ are connected since $f:{\mathbb R}\to {\mathbb R}$ is a Baire-one Darboux function. For every $p=(x,f(x))\in \Gamma_2$ write $B_p=B(p,\frac{|x-x_0|}{2})$. Let $G=\bigcup\limits_{p\in\Gamma_2}B_p$. It is easy to see that there exists $y_0\in {\mathbb R}$ such that $p_0=(x_0,y_0)\in\overline{G}$ and $y_0\ne f(x_0)$. Then $g$ does not have strong Gibson property at $p_0$. Indeed, if $(p_n)_{n=1}^\infty$ be a sequence of points $p_n=(x_n', y_n)\in G$ with $p_n\to p_0$, then $g(p_n)=f(x_n')\to y_0\ne f(x_0)=g(p_0)$, which implies a contradiction.
\end{proof}

Propositions \ref{13} and \ref{14} now leads to the following example.

\begin{example}
  Let $f:{\mathbb R}\to {\mathbb R}$, $f(x)=\sin\frac 1x$ if $x\ne 0$ and $f(x)=1$ if $x=0$, and let $g:{\mathbb R}^2\to {\mathbb R}$, $g(x,y)=f(x)$. Then $g\in {\rm wG}({\mathbb R}^2,{\mathbb R})\setminus {\rm sG}({\mathbb R}^2,{\mathbb R})$.
\end{example}

The following result shows that a strongly Gibson function $f:X\to Y$  can be everywhere discontinuous even in the case where $X=\mathbb R$.
 \begin{proposition}\label{2}
   Let $X$ and $Y$ be topological spaces, $X=A\sqcup B$, where $A$ and $B$ are dense in $X$, $y_1, y_2\in Y$, $y_1\ne y_2$ and $f:X\to Y$ be such a function that $f(x)=y_1$ if $x\in A$, and $f(x)=y_2$ if $x\in B$. Then $f\in {\rm sG}(X,Y)$.
\end{proposition}

 \begin{proof} Fix $x\in X$ and $G\in{\mathcal G}(X)$ such that $x\in {\rm fr}\,G$. Choose an arbitrary open neighborhoods $U$ and $V$ of $x$ and $f(x)$ in $X$ and $Y$, respectively. Since $\overline{A}=\overline{B}=X$ and $U\cap G$ is open set, there exist $x'\in A\cap U\cap G$ and $x''\in B\cap U\cap G$. Then either $f(x')\in V$, or $f(x'')\in V$. Therefore, $f(x)\in\overline{f(U\cap G)}$.
 \end{proof}

\begin{proposition}\label{4}
  Let $X$ be a connected resolvable space, $Y$ be such a $T_1$-space that $|Y|\ge 2$. Then
  $${\rm sG}(X,Y)\setminus {\rm D}(X,Y)\ne\O.$$
\end{proposition}

   \begin{proof} Let $X_1$ and $X_2$ are disjoint dense sets in $X$ such that $X=X_1\cup X_2$. Choose distinct points $y_1,y_2\in Y$ and consider a function  $f:X\to Y$, $f(x)=y_1$ if $x\in X_1$, and $f(x)=y_2$ if $x\in X_2$. According to Proposition \ref{2}, $f\in {\rm sG}(X,Y)$. But $f(X)=\{y_1,y_2\}$ is not connected, then $f\not\in {\rm D}(X,Y)$.
\end{proof}

The following lemma will be useful.

\begin{lemma}\label{1}\cite[p.~136]{Ku2}
  Let $A$ and $B$ be such subsets of topological space $X$ that $A$ is connected, $A\cap B\ne\O$ and $A\setminus B\ne\O$. Then $A\cap{\rm fr}B\ne\O$.
\end{lemma}

\begin{proposition}
  Let $X$ be a regular $T_1$-space and $Y$ be a topological space. Then $${\rm PC}(X,Y)\subseteq {\rm sG}(X,Y).$$
\end{proposition}

\begin{proof}
  Let $f\in {\rm PC}(X,Y)$, $G\in{\mathcal G}(X)$, $x\in {\rm fr}G$. Fix any neighborhoods $U$ and $V$ of $x$ and $f(x)$ in $X$ and $Y$, respectively.  Choose an arbitrary $x_0\in G$. Since $X$ is regular $T_1$-space, there exists an open neighborhood $U_1$ of $x$ such that $\overline{U}_1\subseteq U$ and $x_0\not\in U_1$.
   By peripherally continuity of $f$ at $x$, there exists such an open neighborhood $W$ of $x$ that $W\subseteq U_1$ and $f({\rm fr}\,W)\subseteq V$. Obviously, $G\cap W\ne\O$ and $G\setminus W\ne\O$. Then Lemma \ref{1} implies that $G\cap {\rm fr}W\ne\O$. Take $x'\in G\cap {\rm fr}W$. Then $x'\in G\cap U$ and $f(x')\in V$. Therefore, $f\in {\rm sG}(X,Y)$.
\end{proof}

\begin{lemma}\label{21}
 Let $X$ be a separable resolvable space. Then there exists at most countable set $A\subseteq X$ such that $\overline{A}=\overline{X\setminus A}=X$.
\end{lemma}

\begin{proof} Let $X_1$, $X_2$ and $D$ be dense subsets of $X$ such that $X=X_1\sqcup X_2$ and $|D|\le\aleph_0$. Consider the system
$${\mathcal U}=\{U\in{\mathcal T}(X): U\subseteq \overline{D\cap X_i}\,\,\,\mbox{for an}\,\,\, i\in\{1,2\}\}.$$
We claim that $G=\cup\mathcal U$ is dense in $X$. Indeed, suppose that there exists a nonempty set $V\in{\mathcal T}(X)$ such that $V\cap G=\O$. Then $X_i\cap D$ is nowhere dense in $V$ for $i=1,2$, which implies that $D\cap V=(D\cap X_1\cap V)\cup (D\cap X_2\cap V)$ is nowhere dense in $V$. But $D\cap V$ is dense in $V$ since $D$ is dense in $X$, a contradiction.

Let ${\mathcal V}$ be a maximal subsystem of $\mathcal U$ which consists of disjoint elements. Note that $\cup{\mathcal V}$ is dense in $G$. For every $V\in{\mathcal V}$ there exists $i_V\in\{1,2\}$ such that $X_{i_V}\cap D$ is dense in $V$. Remark that $V\setminus (X_{i_V}\cap D)$ is dense in $V$ too. Put $A=\bigsqcup\limits_{V\in {\mathcal V}}(X_{i_V}\cap D\cap V)$. Since $A$ and $X\setminus A$ are dense in $G$, we have that $\overline{A}=\overline{X\setminus A}=X$. Moreover, $|A|\le \aleph_0$ since $A\subseteq D$.
\end{proof}

\begin{proposition}\label{3} Let $X$ be a separable resolvable Hausdorff space such that a completion to any countable set is a connected set, $Y$ be a $T_1$-space, $|Y|\ge 2$. Then
  $${\rm sG}(X,Y)\setminus {\rm PC}(X,Y)\ne\O.$$
\end{proposition}

\begin{proof}
 It follows from Lemma \ref{21} that there exists at most countable set $A\subseteq X$ such that $\overline{A}=\overline{X\setminus A}=X$.
  Fix two distinct points $y_1,y_2\in Y$ and consider a function $f:X\to Y$, $f(x)=y_1$ if $x\in A$, and $f(x)=y_2$ if $x\in X\setminus A$. According to Proposition \ref{2}, $f\in {\rm sG}(X,Y)$.

We will show that $f\not\in {\rm PC}(X,Y)$. Fix $x\in A$ and let $V=Y\setminus \{y_2\}$. Since $X$ is Hausdorff, there exists an open neighborhood $U$ of $x$ such that $X\setminus \overline{U}\ne\O$. Assume that there exists an open neighborhood $W$ of $x$ such that $W\subseteq U$ and $f({\rm fr}\, W)\subseteq V$. Then ${\rm fr}\, W\subseteq A$, which implies that $|{\rm fr}\, W|\le \aleph_0$. Notice that $X\setminus {\rm fr}\, W=W\sqcup (X\setminus \overline{W})$. The last equality implies a contradiction, because $X\setminus {\rm fr}\, W$ is connected and $W\sqcup (X\setminus \overline{W})$ is not since $W$ and $X\setminus \overline{W}$ are open nonempty sets. Thus, $f$ is not peripherally continuous at $x$.
\end{proof}

\section{On Gibson Baire-one functions}

In this section we show that the Gibson property is equivalent to the almost continuity (in the sense of Husain) and prove that every Gibson Baire-one function from a Baire space  to a metric space is continuous.

\begin{definition}
{\rm  A function $f:X\to Y$ is called {\it almost continuous (in the sense of Husain)} if $f^{-1}(W)\subseteq {\rm int}\overline{f^{-1}(W)}$ for any $W\in{\mathcal T}(Y)$.}
\end{definition}

\begin{proposition}\label{8}
  Let $X$ and $Y$ be topological spaces. Then $f:X\to Y$ is a Gibson function if and only if $f$ is almost continuous.
\end{proposition}

\begin{proof}
  {\sc Necessity.} Let $f:X\to Y$ be a Gibson function, $W\in{\mathcal T}(Y)$ and $x\in f^{-1}(W)$. Denote $G=X\setminus \overline{f^{-1}(W)}$ and assume that $x\in\overline{G}$. Since $f$ is Gibson at $x$, $f(G)\cap W\ne\O$. From the other side, it is easily seen that $f(G)\subseteq Y\setminus W$, which implies a contradiction.

  {\sc Sufficiency.} Choose any $G\in {\mathcal T}(X)$, $x\in {\rm fr}\,G$ and an open neighborhood $W$ of $f(x)$ in $Y$. Denote $U={\rm int}\overline{f^{-1}(W)}$. Since $x\in f^{-1}(W)\subseteq U$, $U\cap G\ne\O$. Notice that $f^{-1}(W)$ is dense in $U\cap G$. Then there exists an $x'\in U\cap G$ such that $f(x')\in W$. Hence, $f$ is Gibson at $x$.
\end{proof}

\begin{proposition}\label{9}
  Let $X$ be a topological space, $Y$ be a regular space, $f:X\to Y$ be a Gibson function with $D(f)\ne\O$. Then there exists an open nonempty set $U\subseteq X$ such that $D(f|_U)=U$.
\end{proposition}

\begin{proof}
  Suppose that $f$ is discontinuous at a point $x_0$. Then there exists  an open neighborhood $W$ of $f(x_0)$ such that $V\cap f^{-1}(Y\setminus W)\ne\O$ for any open neighborhood $V$ of $x_0$ in $X$. It follows from the regularity of $Y$ that there exists  an open neighborhood $W_1$ of $f(x_0)$ such that $\overline{W}_1\subseteq W$. Denote $W_2=Y\setminus \overline{W}$, $G_i={\rm int}\overline{f^{-1}(W_i)}$ for  $i=1,2$ and prove that $U=G_1\cap G_2$ is nonempty. Indeed, according to Proposition \ref{8}, the set $G_1$ is a neighborhood of $x_0$. Then $G_1\cap f^{-1}(W_2)\ne\O$. Moreover, Proposition \ref{8} implies that $f^{-1}(W_2)\subseteq G_2$. Therefore, $U\supseteq G_1\cap f^{-1}(W_2)\ne\O$.

  Now we show that $C(f|_U)=\O$. Assume  $f|_U$ is continuous at $x\in U$. If $f(x)\in W_1$ ($f(x)\in W_2$) then there exists  a neighborhood $V$ of $x$ in $U$ such that $f(V)\subseteq W_1$ ($f(V)\subseteq W_2$), which contradicts the density of  $f^{-1}(W_2)$ ($f^{-1}(W_1)$) in $U$. If $f(x)\in \overline{W}\setminus W_1$ then there exist open neighborhoods  $W'$ of $f(x)$ and $V$ of $x$ in $U$ such that $W'\cap W_1=\O$ and $f(V)\subseteq W'$. But $V\cap f^{-1}(W_1)\ne\O$, a contradiction. Hence, $x\in D(f|_U)$.
\end{proof}

Taking into account that every Baire-one function $f:X\to Y$, where $X$ is a Baire space and $Y$ is a metric space, is pointwise discontinuous, the result below immediately follows from \cite[Theorem 2.1]{Ber} (see also \cite[Corollary 1]{BorsDob}).
We present a proof here for convenience of the reader.

\begin{theorem}
Let $X$ be a Baire space and $Y$ is a metric space. Then every Gibson Baire-one function $f:X\to Y$ is continuous.
\end{theorem}

\begin{proof}
 Assume there exists a Gibson function $f\in {\rm B}_1(X,Y)$ which is discontinuous. Then Proposition \ref{9} implies that the restriction $f|_U$ on an open nonempty set $U\subseteq X$ is everywhere discontinuous. Furthermore, $f|_U$ is of the first Baire class and according to  \cite[p.~405]{Ku1} the set $D(f|_U)$ is of the first category in $U$. From the other hand, $X$ is a Baire space. Then $U$ is of the second category in itself, a contradiction.
\end{proof}

\section{Segmentary connected functions}

Let $X$ be a vector space. For $a, b\in X$ we set
$$[a,b]=\{(1-t) a+t b:t\in[0,1]\}, \qquad [a,b)=\{(1-t) a+t b:t\in[0,1)\}.$$
The set $[a,b]$ is said to be {\it a segment in $X$}.

\begin{definition}
  Let   $Y$ be a topological space. A function $f:X\to Y$ is called {\tt segmentary connected}, $f\in{\rm S}(X,Y)$, if $f([a,b])$ is a connected set for any segment $[a,b]\subseteq X$.
\end{definition}

The definition immediately implies that ${\rm D}(X,Y)\subseteq {\rm S}(X,Y)$ for a topological vector space $X$ and for a topological space $Y$. It is not difficult to verify that ${\rm D}(X,Y)={\rm S}(X,Y)$ if $X$ is a connected subset of $\mathbb R$ and $Y$ is a topological space.

\begin{proposition}\label{15}
  ${\rm S}({\mathbb R}^2,[0,1])\setminus {\rm wG}({\mathbb R}^2,[0,1])\ne\O$.
\end{proposition}

\begin{proof}
  Consider two logarithmic spirals  $S_1\subseteq {\mathbb R}^2$ and $S_2\subseteq {\mathbb R}^2$, which in polar coordinates can be written as $\varrho=e^{\varphi}$ and $\varrho=2e^{\varphi}$, respectively. Set $p_0=(0,0)$. Since $S_1$ and $S_2$ are closed in $X={\mathbb R}^2\setminus\{p_0\}$  and disjoint, there exists such a continuous function $\alpha:X\to[0,1]$ that $S_i=\alpha^{-1}(i-1)$, $i=1,2$. Let $f(p)=\alpha(p)$ if $p\ne p_0$, and $f(p_0)=1$.

  We prove that $f\in {\rm S}({\mathbb R}^2,[0,1])$. Indeed, if $[a,b]$ is such a segment in ${\mathbb R}^2$ that $p_0\not\in[a,b]$, then the restriction $f|_{[a,b]}$ is continuous, consequently, $f([a,b])$ is a connected set. Now let $p_0\in[a,b]$. It is obvious that there exist $p_1\in S_1\cap [a,b]$ and $p_2\in S_2\cap [a,b]$. Since $f(p_1)=0$, $f(p_2)=1$ and $f|_{[p_1,p_2]}$ is continuous, $f([p_1,p_2])=[0,1]$. Taking unto account that $f([p_1,p_2])\subseteq f([a,b])\subseteq [0,1]$, we have $f([a,b])=[0,1]$. Thus, $f$ is segmentary connected.

 We show now that $f\not\in {\rm wG}({\mathbb R}^2,[0,1])$. Let $G_0=f^{-1}([0,\frac 12))$. Choose any point $p\in S_1$ and consider a component $G$ of  $p$ in $G_0$. Then $G\in{\mathcal G}({\mathbb R}^2)$. Since $p_0\in\overline{S}_1$ and $S_1\subseteq G$, $p_0\in \overline{G}$. But $f(p_0)=1\not\in\overline{f(G)}\subseteq [0,\frac 12]$. Therefore, $f$ does not satisfy the weak Gibson property at $p_0$.
\end{proof}

We will need the following two auxiliary facts.

\begin{lemma}\label{12}
  Let $X$ be a topological space, $f:X\to Y$ be a function, $A\subseteq X$, $A=\bigcup\limits_{i\in I}A_i$,  where $f(A_i)$ is connected for every $i\in I$, and for each two points $x_1,x_2\in A$ there exists a finite set of indexes $i_1,\dots,i_n\in I$ such that $x_1\in A_{i_1}$, $x_2\in A_{i_n}$ and $A_{i_k}\cap A_{i_{k+1}}\ne\O$ for $1\le k< n$. Then $f(A)$ is a connected set.
\end{lemma}

\begin{proof}
  It is easy to see that each two points $y_1, y_2\in f(A)$ are contained in a connected subset of $f(A)$. Thus, $f(A)$ is a connected set by  \cite[p.~141]{Ku2}.
\end{proof}

\begin{lemma}\label{7}
  Let $X$ be a locally connected space, $\mathcal B$ be a base of open connected sets in $X$, $Y$ be a topological space and $f:X\to Y$. If $f(U)$ is connected for each $U\in\mathcal B$, then $f\in{\rm wD}(X,Y)$.
\end{lemma}

\begin{proof}
  Choose an arbitrary set $G\in{\mathcal G}(X)$. Then $G=\bigcup\limits_{i\in I} B_i$, where $B_i\in \mathcal B$. According to \cite[p.~144]{Ku2} the sets $B_i$ for $i\in I$ satisfy the condition of Lemma~\ref{12}.  Thus, $f(G)$ is connected.
\end{proof}

\begin{proposition}
  Let $X$ be a topological vector space, $Y$ be a topological space. Then ${\rm S}(X,Y)\subseteq {\rm wD}(X,Y)$.
\end{proposition}

\begin{proof}
  Let $f\in {\rm S}(X,Y)$. Choose a base $\mathcal U$ of balanced open neighborhoods of zero and write ${\mathcal B}=(U+x:x\in X, U\in{\mathcal U})$. Then
$\mathcal B$ is a base of open connected sets in $X$. Take an arbitrary $B\in{\mathcal B}$. Then $B=x+U$ for some $x\in X$ and $U\in{\mathcal U}$. Since $U$ is balanced, $[x,y]\subseteq B$ for each $y\in B$, consequently, $B=\bigcup\limits_{y\in B}[x,y]$. Then $f(B)=\bigcup\limits_{y\in B}f([x,y])$. Since $f([x,y])$ is connected for all $y\in B$ and $\bigcap\limits_{y\in B}f([x,y])\ne\O$, $f(B)$ is connected too. Hence, $f\in {\rm wD}(X,Y)$ by Proposition  \ref{7}.
\end{proof}

Remark that Proposition \ref{4} implies that ${\rm wG}({\mathbb R},Y)\setminus {\rm S}({\mathbb R},Y)\ne\O$ for every $T_1$-space $Y$ with $|Y|\ge 2$. However, the following result shows that if a function $f\in {\rm wG}(X,Y)$ has ``many'' points of continuity, then it is segmentary connected.

\begin{theorem}
  Let $X$ be a topological vector Hausdorff space, $Y$ be a topological space and let $f\in{\rm wG}(X,Y)$ be such a function that $D(f)\cap\ell$ is locally finite for every line $\ell\in X$. Then $f\in {\rm S}(X,Y)$.
\end{theorem}

\begin{proof}
  Fix an arbitrary segment $S=[a,b]\subseteq X$.

   We first consider the case $|D(f)\cap S|\le 1$. If $f|_S$ is continuous, then, evidently, $f(S)$  is connected. Now let $D(f)\cap S=\{x_0\}$.
    Suppose that $f(S)$ is not connected. Then at least one of the sets $f([a,x_0])$ or $f([x_0,b])$ is not connected. Without loss the generality we may assume that $f([a,x_0])$ is not connected. Since $f|_{[a,x_0)}$ is continuous, $f([a,x_0))$ is connected. Then $f(x_0)\not\in \overline{f([a,x_0))}$. Since $X$ is a topological vector Hausdorff space, it is regular, consequently, there are open disjoint sets $W_1$ and $W_2$ in $Y$ such that $f(x_0)\in W_1$ and $f([a,x_0))\in W_2$. The continuity of $f$ on $S\setminus \{x_0\}$ implies that for every $x\in[a,x_0)$ there exists an open connected neighborhood $U_x$ of $x$ in $S$ such that $x_0\not\in U_x$ and $f(U_x)\subseteq W_2$. It is easy to check that $V=\bigcup\limits_{x\in[a,x_0)}U_x$ belong to ${\mathcal G}(X)$ and $x_0\in\overline{V}$. Then, since $f\in{\rm wG}(X,Y)$, we have $f(x_0)\subseteq \overline{f(V)}\subseteq \overline{W_2}$, which implies a contradiction, because $W_1\cap\overline{W_2}=\O$.

Now consider the general case. Then for every $x\in S$ there exist $a_x,b_x\in S$ such that $|D(f)\cap [a_x,b_x]|\le 1$. Let $S_x=[a_x,b_x]$. Then $S=\bigcup\limits_{x\in S}S_x$. Moreover, it follows from what has already been proved that $f(S_x)$ is connected for every $x\in S$. Combining \cite[p.~144]{Ku2} and Lemma \ref{12} gives the connectedness of $f(S)$.
\end{proof}

\section{Baire-one functions with connected graphs}

\begin{theorem}\label{10}
  Let $X$ be a connected and locally connected space, $Y$ be a topological space and let $f:X\to Y$ be a weakly Gibson barely continuous function. Then
\begin{enumerate}
\item $f\in {\rm CG}(X,Y)$;

\item $f\in{\rm wD}(X,Y)$.
\end{enumerate}

\end{theorem}

\begin{proof}
  {\bf (1).} Suppose, contrary to our claim, that ${\rm Gr}(f)=\Gamma_1\sqcup \Gamma_2$, where $\Gamma_1$ and $\Gamma_2$ are closed and open in ${\rm Gr}(f)$. For $i=1,2$ write $X_i=\{x\in X:(x,f(x))\in \Gamma_i\}$. Clearly,  $X=X_1\sqcup X_2$.
   As $X$ is connected we have $F=\overline{X}_1\cap\overline{X}_2$ is nonempty. Since $f$ is barely continuous, there exists $x_0\in C(f|_F)$. We may assume that $x_0\in X_1$. Since $\Gamma_1$ is open in ${\rm Gr}(f)$, there exist open neighborhoods $U$ and $V$ of $x_0$ and $f(x_0)$ in $X$ and $Y$, respectively, such that $(U\times V)\cap {\rm Gr}(f)\subseteq \Gamma_1$. Notice that the continuity of $f|_F$ at $x_0$ and the locally connectedness of $X$ imply that we can choose $U$ such that  $f(U\cap F)\subseteq V$ and $U$ is connected. Then $U\cap F\subseteq X_1$ since ${\rm Gr}(f|_{U\cap F})\subseteq U\times V$.

  Since $x_0\in \overline{X}_2$, $U\cap X_2\ne\O$. Choose any $x'\in U\cap X_2$. Then $x'\not\in F$. Let $G$ be a component of $x'$ in $U\setminus \overline{X}_1$. Notice that $U\cap G\ne\O\ne U\setminus G$. Then, $U\cap{\rm fr}G\ne\O$ by Lemma \ref{1}. It is not difficult to verify that  $U\cap{\rm fr}G\subseteq F$ since $G$ is a component. Fix $x''\in U\cap{\rm fr}G$. Then $f(x'')\in V$.
  But $f\in {\rm wG}(X,Y)$ and $x''\in\overline{G}$, then $f(x'')\in\overline{f(G)}\subseteq\overline{Y\setminus V}=Y\setminus V$ since $f(G)\subseteq Y\setminus V$,  a contradiction.

  {\bf (2).} Let $G\in{\mathcal G}(X)$. Obviously, the restriction $f|_G$ has the weak Gibson property and is barely continuous on $G$. Since $G$ is open in $X$, it is locally connected. Then (1) implies that ${\rm Gr}(f|_G)$ is connected in $G\times f(G)$. If $f(G)=B_1\sqcup B_2$, where $B_i$ is open and closed in $f(G)$ for $i=1,2$, then ${\rm Gr}(f|_G)\subseteq (G\times B_1)\sqcup (G\times B_2)$, where $G\times B_i$ is open and closed in $G\times f(G)$ for $i=1,2$, which contradicts the connectedness of ${\rm Gr}(f|_G)$.
  \end{proof}

\begin{theorem}\label{11}
  Let $X$ be a topological space, $Y$ be a $T_1$-space and $f:X\to Y$ be a barely continuous function. If
  \begin{enumerate}
    \item $X$ is arcwise connected and $f$ is a Darboux function, or

    \item $X$ is a topological vector space and $f$ is segmentary connected,
  \end{enumerate}
    then $f\in {\rm CG}(X,Y)$.
\end{theorem}

\begin{proof}
  Fix $x_0\in X$. In the case (1) for every point $x\in X$, $x\ne x_0$, there exists a homeomorphic embedding $h_x:[0,1]\to X$
  satisfying $h_x(0)=x_0$ and $h_x(1)=x$.  Let $X_0=X\setminus\{x_0\}$ and $A_x=h_x([0,1])$ in the case (1) (and $A_x=[x_0,x]$ in the case (2)) for all $x\in X_0$.  Then $X=\bigcup\limits_{x\in X_0}A_x$. For every $x\in X_0$ we consider a function $g_x=f|_{A_x}$. Then $g_x:A_x\to Y$ is pointwise discontinuous and $g_x\in{\rm D}(A_x,Y)$. By Proposition \ref{6} $g_x\in {\rm wG}(A_x,Y)$.

  Since $A_x$ is homeomorphic to $[0,1]$, it is connected and locally connected. Then Theorem \ref{10}(1) implies that $g_x\in{\rm CG}(A_x,Y)$ for all $x\in X_0$. Notice that ${\rm Gr}(f)=\bigcup\limits_{x\in X_0}{\rm Gr}(g_x)$. For every $x\in X_0$ the set ${\rm Gr}(g_x)$ is connected. Moreover, $(x_0,f(x_0))\in\bigcap\limits_{x\in X_0}{\rm Gr}(g_x)$. Hence, ${\rm Gr}(f)$ is a connected set.
  \end{proof}

 Since every Baire-one function between a hereditarily Baire space and a metric space is barely continuous \cite[p.~125]{Bou}, we have the following corollary from Theorem \ref{10} and Theorem \ref{11}.

  \begin{theorem}\label{5}
  Let $X$ be a hereditarily Baire space, $Y$ be a metric space and \mbox{$f\in {\rm B}_1(X,Y)$}. If

\begin{enumerate}
\item $X$ is a connected and locally connected space and $f\in {\rm wG}(X,Y)$, or

\item $X$ is an arcwise connected space and $f\in {\rm D}(X,Y)$, or

\item $X$ is a topological vector space and $f\in {\rm S}(X,Y)$,

\end{enumerate}
then $f$ has a connected graph.
  \end{theorem}

Remark that a function $f\in{\rm B_1}({\mathbb R}^2,{\mathbb R})$ with a connected graph does not necessarily have either the weak Darboux property, or the weak Gibson property. Indeed, if we consider a function
$f:{\mathbb R}^2\to{\mathbb R}$, $f(x,0)=-x^2$ for $|x|\le1$ and $f(x,y)=0$ for $(x,y)\in {\mathbb R}^2\setminus ([-1,1]\times\{0\})$, then $f$ satisfy the required conditions.

If $X=Y={\mathbb R}$, then the conditions (1), (2) and (3) of Theorem \ref{5} are equivalent. If $X={\mathbb R}^n$ ($n\ge 2$), $Y=\mathbb R$, then $(2)\Rightarrow (1)$ and $(2)\Rightarrow (3)$. Example~\ref{15} gives $(3)\not\Rightarrow (1)$ and, consequently, $(3)\not\Rightarrow (2)$. Proposition \ref{16} below shows, in particular, that $(1)\not\Rightarrow (3)$ and, consequently, $(1)\not\Rightarrow (2)$.

\begin{proposition}\label{16}
  $({\rm sG}({\mathbb R}^2,{\mathbb R})\cap B_1({\mathbb R}^2,{\mathbb R}))\setminus {\rm S}({\mathbb R}^2,{\mathbb R})\ne\O$.
\end{proposition}

\begin{proof}
  Let $(P_n)_{n=1}^\infty$ be a sequence of rectangles $P_n=(a_n,b_n)\times (c_n,d_n)$ with the following properties:
  \begin{enumerate}
    \item $\overline{P}_n\cap \{(x,0):x\in{\mathbb R}\}=\O$ for every $n\in{\mathbb N}$;

    \item $\overline{P}_n\cap \overline{P}_m=\O$ for distinct $m,n\in {\mathbb N}$;

    \item $A=\bigcup\limits_{n=1}^\infty P_n$ is dense in ${\mathbb R}^2$;

    \item $\lim\limits_{n\to\infty}{\rm diam}P_n=0$.
  \end{enumerate}

  We will show that for every nonempty set $G\in{\mathcal G}({\mathbb R}^2)$ there exists such an $n\in{\mathbb N}$ that either $P_n\subseteq G$, or $G\subseteq P_n$. Indeed, fix any nonempty open connected set $G\subseteq{\mathbb R}^2$. If $G\subseteq A$ then choosing $n\in{\mathbb N}$ such that $G\cap P_n\ne\O$, we have $G\subseteq P_n$ since $G$ is connected. Now denote $B={\mathbb R}^2\setminus A$ and assume that $G\cap B\ne\O$. Take $p_0\in G\cap B$ and $\delta>0$ such that $K_1=\{p\in{\mathbb R}^2:d(p,p_0)<2\delta\}\subseteq G$. Let $N_1=\{n\in{\mathbb N}: p_0\not\in \overline{P}_n\,\,\mbox{and}\,\, {\rm diam}\,P_n>\delta\}$. Then $|N_1|<\aleph_0$ and there exists $\delta_1<\delta$ such that $K_2\cap (\bigcup\limits_{n\in N_1}\overline{P}_n)=\O$, where $K_2=\{p\in {\mathbb R}^2: d(p,p_0)<\delta_1\}$.  Denote $N_2=\{n\in {\mathbb N}:p_0\in\overline{P}_n\}$. Remark that $|N_2|\le 1$ according to (2). Then, taking into account that $p_0\not\in\bigcup\limits_{n\in N_2}P_n$, we have $G_1=K_2\setminus (\bigcup\limits_{n\in N_2}\overline{P}_n)$ is nonempty set. Since $B$ is nowhere dense in ${\mathbb R}^2$, there exists such a number $n$ that $P_n\cap G_1\ne\O$. Since $n\not\in N_1\cup N_2$, ${\rm diam}\, P_n\le\delta$. Choose $p_1\in P_n\cap G_1$. Then for every $p\in P_n$
  $$
  d(p,p_0)\le d(p,p_1)+d(p_1,p_0)\le\delta+\delta_1<2\delta.
  $$
  Therefore, $P_n\subseteq K_1\subseteq G$.

  For every $n\in\mathbb N$ we consider a continuous function $f_n:\overline{P}_n\to[0,1]$ such that $f_n(P_n)=(0,1]$ and $f_n({\rm fr}\, P_n)=\{0\}$.
Define a function $f:{\mathbb R}^2\to {\mathbb R}$ by $f(p)=f_n(p)$ if $p\in P_n$, $f(0,0)=1$ and $f(p)=0$ if $p\in B\setminus \{0,0\}$. Notice that the restriction $f|_A$ is continuous and the restriction $f|_B$ is of the first Baire class. Hence, $f\in B_1({\mathbb R}^2,{\mathbb R})$. Moreover, for the segment $S=\{(x,0):x\in [0,1]\}$ we have $f(S)=\{0,1\}$. Thus, $f\not\in {\rm S}({\mathbb R}^2,{\mathbb R})$.

Finally, we will prove that $f\in{\rm sG}({\mathbb R}^2,{\mathbb R})$. Let $p_0\in{\mathbb R}^2$, $G\in{\mathcal G}({\mathbb R}^2)$, $p_0\in\overline{G}$ and $U\subseteq {\mathbb R}^2$ is  an open neighborhood of $p_0$. If $p_0\in A$ then the continuity of $f$ at $p_0$ implies that $f(p_0)\in\overline{f(U\cap G)}$. Now assume $p_0\in B$. If there exists such an $n\in{\mathbb N}$ that $p_0\in\overline{G\cap P_n}$, then $p_0\in\overline{U\cap G\cap P_n}$ and the continuity of $f|_{\overline{P}_n}$ implies that $f(p_0)\in\overline{f(U\cap G\cap P_n)}$, besides, $f(p_0)\in\overline{f(U\cap G)}$.

We will show that  there is $p_1\in U\cap G\cap B$ if $p_0\not\in \overline{G\cap P_n}$ for every $n\in {\mathbb N}$. Choose $\delta>0$ such that $\{p\in{\mathbb R}^2: d(p,p_0)<2\delta\}\subseteq U$. Since $p_0\not\in \overline{G\cap P_n}$ for all $n\in {\mathbb N}$, it follows from (4) that $p_0\not\in \overline{\bigcup\limits_{{\rm diam}(P_n)>\delta}G\cap P_n}$. Taking into account that $p_0\in\overline{G}$, we obtain that the set $\{k\in{\mathbb N}: G\cap P_n\cap U\ne\O\}$ is infinite. Take $n\in{\mathbb N}$ such that ${\rm diam}\, P_n<\delta$ and there is $p\in P_n\cap G$ with $d(p,p_0)<\delta$. Then $\overline{P}_n\subseteq U$. Notice that $P_n\cap G$ is open nonempty subset of the connected set $G$ which is not coincide with $G$. Then $P_n\cap G$ is not closed in $G$. Therefore, there exists $p_1\in {\rm fr}\, P_n\cap G\subseteq U\cap G\cap B$.

Let $G_1$ is a component of $p_1$ in $G\cap U$. Since $p_1\not\in A$, $G_1\not\subseteq P_m$ for every $m\in{\mathbb N}$. Thus, from what has already been proved, it follows that there exists such a number $m$ that $P_m\subseteq G_1$. Then $$f(p_0)\in [0,1]=\overline{f(P_m)}\subseteq \overline{f(G_1)}\subseteq \overline{f(G\cap U)}.$$
Hence, $f$ is strongly Gibson.
\end{proof}

\small{

}

\end{document}